\documentclass[12pt]{amsart}

\usepackage[latin1]{inputenc}
\usepackage{amsmath,amsthm,amsfonts,amscd,amssymb,xspace,color,graphicx,mathtools}
\usepackage[english]{babel}
\usepackage{xr}
\usepackage{mathrsfs}
\usepackage[colorlinks=true, linkcolor=blue, citecolor=blue]{hyperref}

\newtheorem{Prop}{Proposition}[section]
\newtheorem{Lem}[Prop]{Lemma}
\newtheorem{Cor}[Prop]{Corollary}
\newtheorem{Thm}[Prop]{Theorem}

\theoremstyle{remark}

\newtheorem{Ex}[Prop]{Example}

\numberwithin{equation}{section}

\newcommand{\td}[2]{\xrightarrow[#1\rightarrow #2]{}}
\newcommand{\de}{{\rm d}}

\begin{document}

\title{Strong non-arithmeticity for Zariski dense subsemigroups}


\author{Ion Grama}
\author{Jean-Fran\c cois Quint}
\author{Hui Xiao}

\curraddr[Grama, I.]{Univ Bretagne Sud, CNRS UMR 6205, LMBA, Vannes, France.}
\email{ion.grama@univ-ubs.fr}

\curraddr[Quint, J.-F.]{IMAG, Univ Montpellier, CNRS UMR 5149, Montpellier, France.}
\email{Jean-Francois.Quint@umontpellier.fr}

\curraddr[Xiao, H.]{State Key Laboratory of Mathematical Sciences, Academy of Mathematics and Systems Science, Chinese Academy of Sciences, Beijing 100190, China.}
\email{xiaohui@amss.ac.cn}


\begin{abstract}
We prove a non-arithmeticity property for the complex eigenvalues of a Zariski dense subsemigroup of real reductive algebraic groups. 
This property will be used in \cite{GQX} in order to establish a local limit theorem for the operator norm of products of random matrices.
\end{abstract}

\maketitle

\tableofcontents

\section{Introduction}
\label{intro}

Let $\bf G$ be a connected real reductive group, ${\bf A}\subset {\bf G}$ be a maximal $\mathbb R$-split torus, ${\bf Z}\subset{\bf G}$ be the centralizer of ${\bf A}$ and ${\bf P}\supset {\bf Z}$ be a minimal parabolic subgroup defined over $\mathbb R$. The choice of ${\bf P}$ fixes an associated Weyl chamber $\mathfrak a^+\subset \mathfrak a$, where $\mathfrak a$ is the Lie algebra of ${\bf A}(\mathbb R)$. We set $G={\bf G}(\mathbb R)$, $A={\bf A}(\mathbb R)$, $Z={\bf Z}(\mathbb R)$ and $P={\bf P}(\mathbb R)$. 

To this data, one can associate the Jordan projection $\lambda_{\mathbb R}:{\bf G}(\mathbb R)\rightarrow \mathfrak a^+$, which is defined as follows. Any $g$ in $G$ may be written in a unique way as $g=g_hg_eg_u$, where $g_h$ has real and positive eigenvalues (in any rational representation), $g_e$ belongs to a compact subgroup of $G$ and $g_u$ is unipotent, and these three elements commute with each other. Then, $\lambda_{\mathbb R}(g)$ is the unique element of $\mathfrak a^+$ such that $\exp\lambda_{\mathbb R}(g)$ is conjugated to $g_h$ in $G$ (see \cite{Be1}). 

Recall that an element $g$ in $G$ is said to be loxodromic if $\lambda_{\mathbb R}(g)$ belongs to the interior $\mathfrak a^{++}$ of $\mathfrak a^+$. In this case, $g$ is actually semisimple (that is, $g_u=e$) and if $\gamma$ is an element of $G$ such that $\gamma g_h\gamma^{-1}=\exp\lambda_{\mathbb R}(g)$, one has $\gamma g\gamma^{-1}\in Z$. The image of this element in $Z/[Z,Z]$ does not depend on the choice of $\gamma$ (where we have denoted by $[Z,Z]$ the derived subgroup of $Z$, so that $Z/[Z,Z]$ is the Abelianization of $Z$). Henceforth, we denote it by $\lambda(g)$ : this map $\lambda$, from the set $G^{\rm lox}$ of loxodromic elements of $G$ to $Z/[Z,Z]$, is considered in \cite{GR}.

Let $\Gamma\subset G$ be a Zariski dense subsemigroup. Recall that Prasad \cite{P} proved that $\Gamma$ contains loxodromic elements (and actually that $\Gamma\cap G^{\rm lox}$ is still Zariski dense in $G$). The sets $\lambda_\mathbb R(\Gamma)$ and $\lambda(\Gamma\cap G^{\rm lox})$ play a role in the description of certain dynamical and probabilistic phenomena related to $\Gamma$.

We denote by ${\bf S}=[{\bf G},{\bf G}]$ the derived group of ${\bf G}$, by $S={\bf S}(\mathbb R)$ its group of real points and by $\mathfrak s$ the Lie algebra of $S$.
Benoist \cite{Be2} proved that the closed subgroup of $\mathfrak a$ spanned by $\lambda_\mathbb R(\Gamma)$ contains $\mathfrak a\cap\mathfrak s$, which was later used for dynamical purposes by Conze and Guivarc'h \cite{CG}. In \cite{Q1}, it is shown that actually, the elements $\lambda_{\mathbb R}(gh)-\lambda_{\mathbb R}(g)-\lambda_{\mathbb R}(h)$, $g,h\in \Gamma$, span a dense subgroup of $\mathfrak a\cap \mathfrak s$. This technical statement may be used to prove a local limit theorem in linear groups \cite{BQ}. In \cite{GR}, Guivarc'h and Raugi showed that the closed subgroup of $Z/[Z,Z]$ spanned by $\lambda(\Gamma\cap G^{\rm lox})$ is open (which implies Benoist's original result). The purpose of this note is to merge all of these technical results into a single one, which will allow us to extend the range of the local limit theorem for products of random matrices in \cite{GQX}. We shall prove

\begin{Thm} \label{densityJordan}
Let $\Gamma\subset G$ be a Zariski dense subsemigroup. Then, as $(g,h)$ runs among the set of pairs of loxodromic elements of $\Gamma$ such that the product $gh$ is loxodromic, the closed subgroup spanned by the elements $\lambda(gh)\lambda(g)^{-1}\lambda(h)^{-1}$ contains the connected component of the identity of the Lie group $(Z\cap S)/[Z,Z]$.
\end{Thm}

\begin{Ex} Let ${\bf G}={\rm Res}_{\mathbb C/\mathbb R}({\rm SL}_2)$, so that $G={\rm SL}_2(\mathbb C)$, viewed as a real Lie group. An element $g$ in $G$ is loxodromic
if and only if it admits an eigenvalue with modulus $>1$. This eigenvalue $\lambda_1(g)$ is then unique. The theorem says that, if $\Gamma\subset{\rm SL}_2(\mathbb C)$ is a Zariski dense subsemigroup (for the {\em real} Zariski topology), then the set of values of $\lambda_1(gh)\lambda_1(g)^{-1}\lambda_1(h)^{-1}$ spans a dense subgroup of $\mathbb C^\star = \mathbb C \setminus \{0\}$, as $(g,h)$ runs in the set of pairs of loxodromic elements of $\Gamma$ such that the product $gh$ is loxodromic. In the conclusion, dense means dense for the locally compact topology. In \cite{Q1}, it is only shown that, when $(g,h)$ is of the same form, the set of values of $|\lambda_1(gh)\lambda_1(g)^{-1}\lambda_1(h)^{-1}|$ spans a dense subgroup of $\mathbb R^\star_+ = (0, \infty)$.
\end{Ex}

\section{Cross ratios}
\label{crossratios}

We will prove this result by extending the method used in \cite{BQ,Q1}.
This requires us to introduce an appropriate notion of a cross ratio. We
make this precise now. This construction is maybe classical. It follows directly from the properties of the Bruhat decomposition,. 

\subsection{Cross ratios on opposite flag varieties}
\label{flagcrossratios}
Let ${\bf Q}\subset {\bf G}$ be a parabolic subgroup and ${\bf L}\subset{\bf Q}$ be a Levi subgroup. Then, there exists a unique parabolic subgroup ${\bf Q}^-$ of ${\bf G}$ such that ${\bf Q}\cap {\bf Q}^-={\bf L}$. It is called the parabolic subgroup opposite to ${\bf Q}$ with respect to ${\bf L}$. Let ${\bf U}\subset{\bf Q}$ and 
${\bf U}^-\subset{\bf Q}^-$ be the unipotent radicals of ${\bf Q}$ and ${\bf Q}^-$. The Bruhat decomposition says that the product map 
$${\bf U}^-\times {\bf L}\times {\bf U}\rightarrow {\bf G}$$ 
is an isomorphism onto a (Zariski) open subset of ${\bf G}$ (see \cite[Chap. IV]{Bo} or \cite[Chap. X]{Humgp}).
For $g$ in ${\bf U}^- {\bf L} {\bf U}$, we define $\delta(g)$ as being the unique element of ${\bf U}^- g{\bf U}\cap {\bf L}$. The map $\delta:{\bf U}^- {\bf L} {\bf U}\rightarrow {\bf L}$ is regular and, for $g$ in ${\bf G}$ and $\ell$ in ${\bf L}$, one has
\begin{equation}\label{Bruhatequivariance}\delta(\ell g)=\ell\delta(g)\mbox{ and }\delta(g\ell)=\delta(g)\ell.\end{equation}

Let $\mathcal Q$ (resp. $\mathcal Q^-$) be the set of parabolic subgroups of ${\bf G}$ which are conjugated to ${\bf Q}$ (resp. ${\bf Q}^-$); equipped with its natural variety structure, it is called a flag variety of ${\bf G}$. As ${\bf Q}$ and ${\bf Q}^-$ are their own normalizers, one may identify $\mathcal Q$ and $\mathcal Q^-$ with ${\bf G}/{\bf Q}$ and ${\bf G}/{\bf Q}^-$. We will write $\xi_0$ (resp. $\eta_0$)  for ${\bf Q}$ (resp. ${\bf Q}^-$), when seen as an element of $\mathcal Q$ (resp. $\mathcal Q^-$). 

If $\xi$ is in $\mathcal Q$ and $\eta$ is in $\mathcal Q^-$, we will say that they are in general position if the intersection of the associated parabolic subgroups is a Levi subgroup. Equivalently, $\xi$ and $\eta$ are in general position if there exists $g$ in ${\bf G}$ with $\xi=g\xi_0$ and $\eta=g\eta_0$. The set of pairs of elements in general position is Zariski open in $\mathcal Q\times\mathcal Q^-$. The set of $\eta$ in $\mathcal Q^-$ (resp. $\xi$ in $\mathcal Q$) which are in general position with $\xi_0$ (resp. $\eta_0$) is equal to $U\eta_0$ (resp. $U^-\xi_0$).

We denote by $\mathcal W\subset \mathcal Q\times\mathcal Q^-\times \mathcal Q\times\mathcal Q^-$ the set of quadruples $(\xi,\eta,\xi',\eta')$ such that all pairs $(\xi,\eta)$, $(\xi',\eta)$, $(\xi,\eta')$ and $(\xi',\eta')$ are in general position. The cross ratio will be a $G$-invariant regular map
$$(\xi,\eta,\xi',\eta')\mapsto[\xi,\eta,\xi',\eta'], \mathcal W\rightarrow {\bf L}/[{\bf L},{\bf L}].$$

We begin by constructing the cross ratio $[\xi_0,\eta_0,\xi,\eta]$: thus we take $\xi$ in $\mathcal Q$ and $\eta$ in $\mathcal Q^-$ such that $\xi$ is in general position with $\eta_0$ and $\eta$, and $\eta$ is in general position with $\xi_0$. As $\xi$ and $\eta$ are in general position, there exists $g$ in ${\bf G}$ such that 
$\xi=g\xi_0$ and $\eta=g\eta_0$. Since $g\xi_0$ is in general position with $\eta_0$, we have $g\xi_0\in {\bf U}^-\xi_0$, that is, $g\in {\bf U}^-{\bf Q}={\bf U}^-{\bf L}{\bf U}$. In the same way, since $g\eta_0$ is in general position with $\xi_0$, we have $g\eta_0\in {\bf U}\eta_0$, that is, $g\in {\bf U}{\bf Q}^-$, hence $g^{-1}\in{\bf U}^-{\bf L}{\bf U}$. Thus, the elements $\delta(g)$ and $\delta(g^{-1})$ coming from the Bruhat decomposition are well defined in ${\bf L}$. We will define 
$[\xi_0,\eta_0,\xi,\eta]=[\xi_0,\eta_0,g\xi_0,g\eta_0]$ to be the image in ${\bf L}/[{\bf L},{\bf L}]$ of $\delta(g)\delta(g^{-1})$. This is possible thanks to

\begin{Lem} Let $\xi \in \mathcal Q$ and $\eta \in \mathcal Q^-$ be such that $(\xi_0,\eta_0,\xi,\eta)$ belongs to $\mathcal W$. Fix $g$ in ${\bf G}$ such that 
$\xi=g\xi_0$ and $\eta=g\eta_0$. Then, the element $\delta(g)\delta(g^{-1})$ does not depend on the choice of $g$.
\end{Lem}

\begin{proof} Let $g'$ in ${\bf G}$ be such that $\xi=g'\xi_0$ and $\eta=g'\eta_0$. We have $g^{-1}g'\in {\bf Q}\cap{\bf Q}^-={\bf L}$. Set $\ell=g^{-1}g'$. From \eqref{Bruhatequivariance}, we get
$$\delta(g')\delta((g')^{-1})=\delta(g\ell)\delta(\ell^{-1}g^{-1})=\delta(g)\ell\ell^{-1}\delta(g^{-1})=\delta(g)\delta(g^{-1})$$
as required.
\end{proof}

As announced we can therefore set $[\xi_0,\eta_0,\xi,\eta]=[\xi_0,\eta_0,g\xi_0,g\eta_0]$ as being the image of $\delta(g)\delta(g^{-1})$ in ${\bf L}/[{\bf L},{\bf L}]$. We work only in the Abelianization of ${\bf L}$ because we need to define the cross ratio of any quadruple in $\mathcal W$ (which we want to be ${\bf G}$-invariant).

Indeed, there is no choice left: for $(\xi,\eta,\xi',\eta')$ in $\mathcal W$, we choose $g$ in ${\bf G}$ such that $\xi=g\xi_0$ and $\eta=g\eta_0$. We then would like to define the cross ratio $[\xi,\eta,\xi',\eta']$ as being equal to $[\xi_0,\eta_0,g^{-1}\xi',g^{-1}\eta']$. Now, this is possible again thanks to the properties of the $\delta$ map, as claimed in the following

\begin{Lem} Let $(\xi,\eta,\xi',\eta')$ be in $\mathcal W$. Then 
the element $$[\xi_0,\eta_0,g^{-1}\xi',g^{-1}\eta']\in{\bf L}/[{\bf L},{\bf L}]$$ 
does not depend on the choice of a $g$ in ${\bf G}$ such that $\xi=g\xi_0$ and $\eta=g\eta_0$.
\end{Lem}

\begin{proof} Again, if $g'$ is another element of ${\bf G}$ such that $\xi=g'\xi_0$ and $\eta=g'\eta_0$, we have $g'=g\ell$ for some $\ell$ in ${\bf L}$, and hence, we need to compare $[\xi_0,\eta_0,g^{-1}\xi',g^{-1}\eta']$ with $[\xi_0,\eta_0,\ell^{-1}g^{-1}\xi',\ell^{-1}g^{-1}\eta']$. Choose $h$ in ${\bf G}$ with 
$g^{-1}\xi'=h\xi_0$ and $g^{-1}\eta'=h\eta_0$. On one hand, by definition, we have
$$[\xi_0,\eta_0,g^{-1}\xi',g^{-1}\eta']=\delta(h)\delta(h^{-1})[{\bf L},{\bf L}].$$
On the other hand, we have,
$\ell^{-1}g^{-1}\xi'=\ell^{-1}h\xi_0$ and $\ell^{-1}g^{-1}\eta'=\ell^{-1}h\eta_0$ and hence, thanks to \eqref{Bruhatequivariance},
\begin{multline*}[\xi_0,\eta_0,\ell^{-1}g^{-1}\xi',\ell^{-1}g^{-1}\eta']=\delta(\ell^{-1}h)\delta(h^{-1}\ell)[{\bf L},{\bf L}]\\
=\ell^{-1}\delta(h)\delta(h^{-1})\ell[{\bf L},{\bf L}]
=\delta(h)\delta(h^{-1})[{\bf L},{\bf L}].\end{multline*}
The lemma follows.
\end{proof}

As a consequence, for $(\xi,\eta,\xi',\eta')$ in $\mathcal W$, we define the cross ratio by $[\xi,\eta,\xi',\eta']=[\xi_0,\eta_0,g^{-1}\xi',g^{-1}\eta']$ where $g$ is in ${\bf G}$ and $\xi=g\xi_0$ and $\eta=g\eta_0$. 
The terminology is justified by the following

\begin{Ex} Recall that the usual cross ratio is defined, for any four elements $s,t,u,v$ in $\mathbb P^1_{\mathbb C}=\mathbb C\cup\{\infty\}$, with $s\neq t$, $s\neq v$, $t\neq u$, $u\neq v$, by 
$$[s,t,u,v]=\frac{(s-v)(u-t)}{(s-t)(u-v)}$$
(with the usual meaning when some of $s,t,u,v$ are $\infty$).

Let ${\bf G}={\rm GL}_2$ be the general linear group of invertible $2 \times 2$ matrices. We choose ${\bf Q}$ to be the group of upper triangular matrices and ${\bf L}$ to be the group of diagonal matrices. Then, ${\bf Q}^-$ is the group of lower triangular matrices. The variety $\mathcal Q=\mathcal Q^-$ may be identified with $\mathbb P^1_{\mathbb C}$ by the map which sends $z$ in $\mathbb C$ to the stabilizer in ${\rm GL}_2(\mathbb C)$ of the vector line $\mathbb C(z,1)$ and $\infty$ to 
the stabilizer in ${\rm GL}_2(\mathbb C)$ of the vector line $\mathbb C(1,0)$. Saying that two elements of $\mathbb P^1_{\mathbb C}$ are in general position (viewed as flags) is saying that they are different. Then, we claim that, for 
$s,t,u,v$ in $\mathbb P^1_{\mathbb C}=\mathbb C\cup\{\infty\}$, with $s\neq t$, $s\neq v$, $t\neq u$, $u\neq v$, the cross ratio built above is just the element of ${\bf L}$, that is, the diagonal matrix given by
$$\begin{pmatrix}[s,t,u,v]& 0\\ 0 & [s,t,u,v]^{-1}\end{pmatrix}.$$

Indeed, as both cross ratio maps are ${\rm GL}_2(\mathbb C)$-invariant, it suffices to prove this formula when $s=\xi_0=\infty$ and $t=\eta_0=0$. Then, we need to compute $\delta(g)\delta(g^{-1})$ where $g$ is in ${\rm GL}_2(\mathbb C)$ and $u=g\infty$, $v=g 0$. We can take $g=\begin{pmatrix} u&v\\ 1&1\end{pmatrix}$, which yields 
$$\delta(g)=\begin{pmatrix} u&0\\ 0&\frac{u-v}{u}\end{pmatrix}\mbox{ and }\delta(g^{-1})=\begin{pmatrix} \frac{1}{u-v}&0\\ 0&1\end{pmatrix}$$
hence 
$$\delta(g)\delta(g^{-1})= \begin{pmatrix} \frac{u}{u-v}&0\\ 0&\frac{u-v}{u}\end{pmatrix}=\begin{pmatrix}[\infty,0,u,v]& 0\\ 0 & [\infty,0,u,v]^{-1}\end{pmatrix}$$
as claimed.
\end{Ex}

\subsection{Real cross ratios} 
\label{realcrossratio}
We now specialize this construction for the minimal $\mathbb R$-parabolic subgroups. We let ${\bf Q}$ be the minimal parabolic subgroup ${\bf P}$ and we take ${\bf L}$ to be equal to ${\bf Z}$. We also write $\mathcal P=\mathcal Q=\mathcal Q^-$ for the maximal real flag manifold of ${\bf G}$. We denote the space of real points of $\mathcal P$ by $\mathcal P_{\mathbb R}$: this is a set of minimal parabolic $\mathbb R$-subgroups of ${\bf G}$ and it may be identified with $G/P$ (see \cite[Chap. V]{Bo}).

If $\xi,\eta,\xi',\eta'$ in $\mathcal P_\mathbb R$ are such that $(\xi,\eta,\xi',\eta')\in\mathcal W$, by construction, the cross ratio $[\xi,\eta,\xi',\eta']$ belongs to ${\bf Z}({\mathbb R})/([{\bf Z},{\bf Z}])(\mathbb R)$. This group is the same as $Z/[Z,Z]$:

\begin{Lem}\label{goodAbelian} The natural map $[Z,Z]\rightarrow ([{\bf Z},{\bf Z}])(\mathbb R)$ is an isomorphism.
\end{Lem}

\begin{proof} The group ${\bf Z}$ being the centralizer of a maximal split torus, it is connected and the group $[{\bf Z},{\bf Z}]$ is connected and $\mathbb R$-ani\-so\-tro\-pic. Therefore, the group $([{\bf Z},{\bf Z}])(\mathbb R)$ is a compact connected semisimple Lie group. Since the closed subgroup $[Z,Z]\subset ([{\bf Z},{\bf Z}])(\mathbb R)$ has the same Lie algebra, we obtain $([{\bf Z},{\bf Z}])(\mathbb R)=[Z,Z]$.
\end{proof}

From now on, we consider $Z/[Z,Z]$ as a subset of $({\bf Z}/[{\bf Z},{\bf Z}])(\mathbb R)$. By construction, for $\xi,\eta,\xi',\eta'$ as above, the cross ratio $[\xi,\eta,\xi',\eta']$ belongs to $Z/[Z,Z]$.

If $g$ is a loxodromic element in $G$, we denote by $\xi_g^+$ the attractive fixed point of $g$ in $\mathcal P_{\mathbb R}$: this is the flag $h\xi_0$ where $h$ is such that $hg_{h}h^{-1}$ belongs to $\exp(\mathfrak a^{++})$. We also denote by $\xi_g^-$ the attractive fixed point of $g^{-1}$.

A key ingredient towards Theorem \ref{densityJordan} is the following extension of \cite[Cor. 7.11]{BQ}:

\begin{Prop}\label{asymptoticcrossratio}
Let $g$ and $h$ be loxodromic elements in $G$ such that $(\xi_g^+,\xi_g^-,\xi_h^+,\xi_h^-)\in\mathcal W$. Then, for sufficiently large $m$ and $n$, the element $g^mh^n$ is loxodromic and we have the convergence in $Z/[Z,Z]$:
$$\lambda(g^m h^n)\lambda(g)^{-m}\lambda(h)^{-n}\td{m,n}{\infty}[\xi_g^+,\xi_g^-,\xi_h^+,\xi_h^-].$$
\end{Prop}

The proof of this result will occupy the next three sections.

\section{Product of proximal endomorphisms}
\label{secproxprod}

Here we slightly adapt the language of \cite[Sect. 7.3]{BQ} in order to analyze not only the spectral radius of a proximal endomorphism, but all of the leading eigenvalue, that is, we take into account the complex argument of this eigenvalue. This will yield Lemma \ref{asymptoticcrossratio1} below, which is a first step towards Proposition \ref{asymptoticcrossratio}. 

Let $V$ be a finite-dimensional complex vector space and let $g$ be a linear endomorphism of $V$. We say that $g$ is proximal if $g$ admits a unique eigenvalue with maximal modulus and the multiplicity of this eigenvalue $\lambda_1(g)$ is one: thus, the generalized eigenspace associated with $\lambda_1(g)$ has dimension $1$ (and is therefore just the eigenspace associated with $\lambda_1(g)$). We denote this line by $V_g^+$ and we write $V_g^<$ for the unique $g$-invariant complementary hyperplane.
By Rouch\'e's formula, the set of proximal endomorphisms is an open subset ${\rm Prox}(V)$ of ${\rm End}(V)$ and the function $\lambda_1$ is holomorphic on 
${\rm Prox}(V)$. In case $V$ has dimension $1$, we also ask that $g\neq 0$ for $g$ to be proximal. 

Let $X$ be in $\mathbb P(V)$ and $Y$ be in $\mathbb P(V^*)$ (where $V^*$ is the dual space of $V$; we will also identify $\mathbb P(V^*)$ with the space of hyperplanes of $V$). Then, $X$ and $Y$ are in general position in the sense of Subsection \ref{flagcrossratios} if and only if they are not orthogonal, that is, if for $v\neq 0$ in $X$ and $\varphi\neq 0$ in $Y$, one has $\varphi(v)\neq 0$.
If $X,X' \in \mathbb P(V)$ and $Y,Y' \in \mathbb P(V^*)$ are such that $(X,Y)$, $(X',Y)$, $(X,Y')$ and $(X',Y')$ are in general position, 
we will say that the quadruple $(X,Y,X',Y')$ is in general position and we will write 
$$[X,Y,X',Y']=\frac{\varphi(v')\varphi'(v)}{\varphi(v)\varphi'(v')},$$ 
where $v\neq 0$ is in $X$, $v'\neq 0$ is in $X'$, $\varphi\neq 0$ is in $Y$ and $\varphi'\neq 0$ is in $Y'$. This is actually a cross ratio in the sense of Subsection 
\ref{flagcrossratios}.

The following lemma is more or less the same as \cite[Lemma 7.5]{BQ}:

\begin{Lem} \label{asymptoticcrossratio1}
Let $g,h$ be proximal endomorphisms of $V$ and assume that the quadruple $(V^+_g,V_g^<,V_h^+,V_h^<)$ is in general position. 
Then, for sufficiently large $m$ and $n$, $g^mh^n$ is proximal and we have
$$\lambda_1(g^mh^n)\lambda_1(g)^{-m}\lambda_1(h)^{-n}\td{m,n}{\infty}[V^+_g,V_g^<,V_h^+,V_h^<].$$
\end{Lem}

\begin{proof} If $g$ and $h$ are rank-one endomorphisms of $V$, a direct computation shows that $gh$ is proximal and 
$$\lambda_1(gh)=\lambda_1(g)\lambda_1(h)[V^+_g,V_g^<,V_h^+,V_h^<].$$
This implies the general case since, as $m,n\rightarrow \infty$, the endomorphisms $\lambda_1(g)^{-m}g^m$ and $\lambda_1(h)^{-n}h^n$ both converge to rank-one  endomorphisms.
\end{proof}

\section{Proximal representations}
\label{proxrep}

Following the strategy of \cite[Chap. 7]{BQ}, we will deduce Proposition \ref{asymptoticcrossratio} 
 from applying Lemma \ref{asymptoticcrossratio1} in sufficiently many representations of ${\bf G}$. The difficulty here is that we have to choose these representations more carefully. In the present section, we discuss why these choices are possible.

We recall some facts from representation theory (see \cite[Chap. 31]{Humgp}). 
If ${\bf H }$ is an algebraic group, we denote its group of characters by ${\rm X}({\bf H})$. To make computations easier, we write the group law of ${\rm X}({\bf H})$ additively.

Let ${\bf T}$ be a maximal torus of ${\bf G}$. We let $\Sigma_a\subset {\rm X}({\bf T})$ be the set of absolute roots, that is, the set of weights of the adjoint action of ${\bf T})$ on the Lie algebra of ${\bf G}$, and we choose a positive system $\Sigma_a^+\subset\Sigma_a$. We write $\Pi_a\subset \Sigma_a^+$ for the associated basis. Set $W$ to be the (absolute) Weyl group, that is, the finite group of automorphisms of ${\rm X}({\bf T})$ obtained by the action of the normalizer of ${\bf T}$ in ${\bf G}$. 

We equip the group ${\rm X}({\bf T})$ with the partial order associated with the subsemigroup $\bigoplus_{\alpha\in\Pi_a}\mathbb N\alpha$. 
In other words, for $\chi,\chi'$ in ${\rm X}({\bf T})$, we write $\chi\geq \chi'$ if $\chi-\chi'$ is a sum of elements of $\Pi_a$.
A character $\chi$ in ${\rm X}({\bf T})$ is said to be dominant if, for any $w$ in $W$, one has $\chi\geq w\chi$. Any $W$-orbit contains a unique dominant character. 

Let $\rho:{\bf G}\rightarrow {\rm GL}(V)$ be an irreducible rational representation, where $V$ is a finite-dimensional complex vector space. Then, the space $V^{\bf N}$ of ${\bf N}$-invariant vectors in $V$ is a line. This line is a weight space for ${\bf T}$, for a character $\chi$ which is dominant. 
If $\chi'$ is another weight of ${\bf T}$ in $V$, then $\chi-\chi'$ is a sum of elements of $\Pi_a$. The character $\chi$ is called the highest weight of the representation and for any dominant weight $\chi$ of ${\bf T}$, there exists a unique irreducible representation of ${\bf G}$ with highest weight $\chi$ (up to isomorphism). 

Let ${\bf B}\subset{\bf G}$ be the Borel subgroup containing ${\bf T}$ associated with the choice of the positive system $\Sigma_a^+$: the Lie algebra of the unipotent radical ${\bf N}$ of ${\bf B}$ is the sum of the root spaces associated with the elements of $\Sigma_a^+$.

Let $\rho:{\bf G}\rightarrow {\rm GL}(V)$ be an irreducible rational representation, where $V$ is a finite-dimensional complex vector space. Then, the space $V^{\bf N}$ of ${\bf N}$-invariant vectors in $V$ is a line. This line is a weight space for ${\bf T}$, for a character $\chi$ which is dominant. 
If $\chi'$ is another weight of ${\bf T}$ in $V$, then $\chi\geq \chi'$. The character $\chi$ is called the highest weight of the representation and for any dominant weight $\chi$ of ${\bf T}$, there exists a unique irreducible representation of ${\bf G}$ with highest weight $\chi$ (up to isomorphism). 

Now, we choose a parabolic subgroup ${\bf Q}\supset{\bf B}$ and we let ${\bf V}$ be the unipotent radical of ${\bf Q}$ and ${\bf L}$ be its Levi subgroup containing ${\bf T}$, so that ${\bf Q}={\bf L}{\bf V}$. The choice of ${\bf Q}$ is completely determined by the choice of a subset $\Theta\subset\Pi_a$: the Lie algebra of ${\bf L}$ is the sum of the Lie algebra of ${\bf T}$ and of the root spaces associated with roots in the set $\langle\Theta\rangle$ of roots in the space spanned by $\Theta$; the Lie algebra of ${\bf V}$ is the sum of the root spaces associated with the roots in $\Sigma_a^+\setminus\langle\Theta\rangle$. 

As ${\bf T}$ is a maximal torus of ${\bf L}$, the restriction morphism
${\rm X}({\bf L})\rightarrow{\rm X}({\bf T})$ is injective and we may consider ${\rm X}({\bf L})={\rm X}({\bf Q})$ as a subgroup of ${\rm X}({\bf T})$. 

Suppose now ${\bf G}$ is semisimple. Then $\Pi_a$ spans a finite index subgroup of ${\rm X}({\bf T})$.
Set ${\rm X}_{\mathbb R}({\bf T})=\mathbb R\otimes_{\mathbb Z}{\rm X}({\bf T})$. We equip ${\rm X}_{\mathbb R}({\bf T})$ with a $W$-invariant scalar product $(.,.)$. For $\alpha$ in $\Pi_a$, we let $\chi_\alpha$ be the associated fundamental weight, that is, the unique element of the subspace of ${\rm X}_{\mathbb R}({\bf T})$ such that $2(\chi_\alpha,\alpha)=(\alpha,\alpha)$ and, for $\beta\neq \alpha$ in $\Pi_a$, $(\chi_\alpha,\beta)=0$. 
By \cite[Sect. 31.1]{Humgp}, we have 
${\rm X}({\bf T})\subset \bigoplus_{\alpha\in\Pi_a}\mathbb Z\chi_{\alpha}$.
A character $\chi$ of ${\bf T}$ is dominant if it belongs to 
$\bigoplus_{\alpha\in\Pi_a}\mathbb N\chi_{\alpha}$. 

The following description of the character group of ${\bf L}$ should be standard. 

\begin{Prop} \label{Levicharacter1}
Assume ${\bf G}$ is semisimple. The character group of ${\bf L}$ may be written as follows:
\begin{equation}\label{chargroupformula}
{\rm X}({\bf L})={\rm X}({\bf T})\cap \left(\bigoplus_{\alpha\in\Pi_a\setminus\Theta}\mathbb Z\chi_{\alpha}\right).\end{equation}
\end{Prop}

\begin{proof}
Let $\chi$ be in ${\rm X}({\bf L})$. For $\alpha\in \Theta$, as the associated weight space is contained in the Lie algebra of ${\bf L}$, by \cite[Sect. 13.18]{Bo} or \cite[Sect. 24.1]{Humgp}, there exists an element $n_\alpha$ of the normalizer of ${\bf T}$ in ${\bf L}$ which induces on ${\rm X}_{\mathbb R}({\bf T})$ the reflection $\sigma_\alpha$ associated with $\alpha$. Hence, for $t$ in ${\bf T}$, we get 
$$\sigma_\alpha(\chi)(t)=\chi(n_\alpha tn_\alpha^{-1})=\chi(n_\alpha)\chi(t)\chi(n_\alpha)^{-1}=\chi(t).$$
Thus, $\sigma_\alpha(\chi)=\chi$ which is to say that $(\alpha,\chi)=0$ and hence that $\chi$ belongs to $\bigoplus_{\alpha\in\Pi_a\setminus\Theta}\mathbb Z\chi_{\alpha}$. Thus, we have proved that the left-hand side of \eqref{chargroupformula} is contained in the right-hand side.

Conversely, we recall the argument which may essentially be found in \cite[Sect. 31.4]{Humgp}. 

First, for $\alpha$ in $\Pi_a\smallsetminus\Theta$, we claim that there exists an integer $d>0$ and a character $\chi$ of ${\bf L}$ whose restriction to ${\bf T}$ is $\chi_\alpha^d$. Indeed, by Chevalley's Theorem (see \cite[Sect. 5.1]{Bo}, \cite[Sect. 11.2]{Humgp}), there exists a complex representation $\rho:{\bf G}\rightarrow{\rm GL}(V)$ and a line $X\subset V$ whose stabilizer in ${\bf G}$ is the standard parabolic subgroup ${\bf P}_\alpha\supset {\bf B}$ associated with the subset $(\Pi_a\setminus\{\alpha\})\subset \Pi_a$. As we assumed $\alpha\notin\Theta$, we have ${\bf Q}\subset{\bf P}_\alpha$. The action of ${\bf P}_\alpha$ on $X$ is by a character $\chi$. Since $X$ is ${\bf N}$-invariant, by \cite[Sect. 31.2]{Humgp}, up to replacing $V$ by the subspace spanned by ${\bf G}X$, we may assume that $V$ is irreducible, with highest weight $\chi$.

As above, for $\beta$ in $\Pi_a\setminus\{\alpha\}$ the normalizer of ${\bf T}$ in ${\bf P}_\alpha$ contains an element $n_\beta$ which acts by $\sigma_\beta$ on ${\rm X}_{\mathbb R}({\bf T})$ and hence $(\chi,\beta)=0$. Therefore, we get $\chi=d\chi_\alpha$, for some integer $d$. Since ${\bf G}/{\bf P}_{\alpha} $ is a projective variety (not reduced to a point), we have $d\neq 0$. Since $\chi$ is dominant, we have $d>0$ (but we actually only need to know that $d\neq 0$). 

Now, we have proved that the group ${\rm X}({\bf L})\subset\bigoplus_{\alpha\in\Pi_a\setminus\Theta}\mathbb Z\chi_{\alpha}$ contains a nonzero multiple of $\chi_\alpha$, for each $\alpha$ in $\Pi_a\setminus\Theta$. Therefore it contains the group $\bigoplus_{\alpha\in\Pi_a\setminus\Theta}d\mathbb Z\chi_{\alpha}$ for some $d>0$ sufficiently large. Let $\chi$ be in ${\rm X}({\bf T})\cap \left(\bigoplus_{\alpha\in\Pi_a\setminus\Theta}\mathbb Z\chi_{\alpha}\right)$. Let ${\bf N}^-$ be the unipotent radical of the Borel subgroup of ${\bf G}$ opposite to ${\bf B}$ with respect to ${\bf T}$. For $g$ in ${\bf L}\cap {\bf N^-}{\bf T}{\bf N}$, set $\chi(g)=\chi(t)$, where $t$ is the torus component of the Bruhat decomposition of $g$. Then, $\chi$ is a rational function on ${\bf L}$ and $\chi^d$ is a regular function on ${\bf L}$. As ${\bf L}$ is smooth, it is normal, and hence $\chi$ is actually a regular function on ${\bf L}$. We claim that $\chi$ is a character of ${\bf L}$. 

As $d\chi$ is a character, for $g$ in ${\bf L}$, we have $\chi(g)^d\neq 0$, hence $\chi(g)\neq 0$. Thus, the function $f$ defined on ${\bf L}\times{\bf L}$ by 
$$f(g_1,g_2)=\chi(g_1g_2)\chi(g_1)^{-1}\chi(g_2)^{-1}$$ is regular. Since $d\chi$ is a character, we get $f^d=1$. Hence, as ${\bf L}$ is connected, $f$ is constant. As $f$ is equal to $1$ on ${\bf T}\times{\bf T}$, we get $f=1$ and $\chi$ is indeed a character. We have shown that the right-hand side of \eqref{chargroupformula} is contained in the left-hand side.
\end{proof}

We go back to the general case, where we do not assume anymore ${\bf G}$ to be semisimple but only to be reductive.
If $\rho:{\bf G}\rightarrow {\rm GL}(V)$ is an irreducible representation of ${\bf G}$ with highest weight $\chi$, let us say that $\rho$ is ${\bf Q}$-proximal if  the highest weight space $V_{\chi}=V^{\bf N}$ is actually ${\bf Q}$-invariant. Then, in particular, the highest weight $\chi$ belongs to ${\rm X}({\bf L})$.

\begin{Cor} \label{Levicharacter2}
An irreducible representation  $\rho:{\bf G}\rightarrow {\rm GL}(V)$ is ${\bf Q}$-proximal if and only if its highest weight $\chi$ belongs to ${\rm X}({\bf L})$. The set of highest weights of ${\bf Q}$-proximal representations spans the group ${\rm X}({\bf L})$.
\end{Cor}

\begin{proof} The first statement directly follows from the construction of the irreducible representation with a given highest weight (see \cite[Sect. 31.4]{Humgp}). The second now follows from Proposition \ref{Levicharacter1}. 

More precisely, we can assume that ${\bf G}$ is semisimple. Indeed, if ${\bf S}=[{\bf G},{\bf G}]$ is the derived group of ${\bf G}$ and ${\bf C}$ is its center, we have ${\bf G}={\bf S}{\bf C}$ and ${\bf C}\subset {\bf T}$, so that a character $\chi$ in ${\rm X}({\bf T})$ is in ${\rm X}({\bf L})$ if and only $\chi_{|{\bf S}\cap{\bf T}}$ is in ${\rm X}({\bf S}\cap{\bf L})$. 

In case ${\bf G}$ is semisimple, Proposition \ref{Levicharacter1} we have ${\rm X}({\bf L})={\rm X}({\bf T})\cap \left(\bigoplus_{\alpha\in\Pi_a\setminus\Theta}\mathbb Z\chi_{\alpha}\right)$, whereas the first statement says that the set of highest weights of ${\bf Q}$-proximal representations is ${\rm X}({\bf T})\cap \left(\bigoplus_{\alpha\in\Pi_a\setminus\Theta}\mathbb N\chi_{\alpha}\right)$. The conclusion follows.
\end{proof}

\section{Cross ratios and loxodromic elements}

We will now relate the abstract cross ratios of Subsection \ref{flagcrossratios} with the cross ratios in projective spaces of Section \ref{secproxprod}. This will allow us to deduce Proposition \ref{asymptoticcrossratio} from Lemma \ref{asymptoticcrossratio1}.

\subsection{Proximal representations and cross ratios} We begin by computing the image of a cross ratio by a dominant character.

We keep the previously introduced notation. 
In particular, $\bf Q$ and $\bf Q^-$ are opposite parabolic subgroups of the complex reductive algebraic group $\bf G$. 
Given $\rho:{\bf G}\rightarrow {\rm GL}(V)$, a ${\bf Q}$-proximal irreducible representation of ${\bf G}$, the highest weight space $V^{\bf N}$ is ${\bf Q}$-invariant. The orbit map $g\mapsto gV^{\bf N}$ induces a morphism of varieties $\mathcal Q={\bf G}/{\bf Q}\rightarrow\mathbb P(V)$. By abuse of notation, we still denote this map by $\rho$. In the same way, considering the contragredient representation $\rho^\vee:{\bf G}\rightarrow {\rm GL}(V^*)$, we get a map $\rho^\vee:\mathcal Q^- \rightarrow \mathbb P(V^*)$, where $\mathcal Q^- = \bf G / Q^-$. 

\begin{Lem}\label{representcrossratio}
Let $\rho:{\bf G}\rightarrow {\rm GL}(V)$ be a ${\bf Q}$-proximal irreducible representation of ${\bf G}$ with highest weight $\chi$. 
Then, for $(\xi,\eta,\xi',\eta')$ in $\mathcal W$, we have
\begin{equation}\label{crossratiorepformula}\chi([\xi,\eta,\xi',\eta'])=[\rho(\xi),\rho^\vee(\eta),\rho(\xi'),\rho^\vee(\eta')].\end{equation}
\end{Lem}

\begin{proof} Since both sides of \eqref{crossratiorepformula} define ${\bf G}$-invariant functions $\mathcal W\rightarrow \mathbb C^*$, it suffices to prove the formula when $\xi$ and $\eta$ are respectively the fixed points $\xi_0$ and $\eta_0$ of ${\bf Q}$ and ${\bf Q}^-$ in $\mathcal Q$. 

Take $v\neq 0$ in $\rho(\xi_0)$ and $\varphi\neq 0$ in $\rho^\vee(\eta_0)$. Note that $v$ is $\rho({\bf V})$-invariant and $\varphi$ is $\rho^\vee({\bf V}^-)$-invariant. Also, as $\rho(\xi_0)$ is the highest weight space, for $\ell$ in ${\bf L}$, we have $\rho(\ell)v=\chi(\ell) v$.

Now, choose $g$ in ${\bf G}$ with $g\xi_0=\xi'$ and $g\eta_0=\eta'$. Write the Bruhat decomposition of $g$ as $g=u_1^-\delta(g)u_1$ and that of $g^{-1}$ as 
$g^{-1}=u_2^-\delta(g^{-1})u_2$. On one hand, by the definition in Subsection \ref{flagcrossratios}, we have
$$\chi([\xi_0,\eta_0,g\xi_0,g\eta_0])=\chi(\delta(g)\delta(g^{-1})).$$
On the other hand, by the definition in Section \ref{secproxprod}, we have
\begin{align*}[\rho(\xi_0),\rho^\vee(\eta_0),\rho(g\xi_0),\rho^\vee(g\eta_0)]
&=\frac{\varphi(\rho(g)v)(\rho^\vee(g)\varphi)(v)}{\varphi(v)(\rho^\vee(g)\varphi)(\rho(g)v)}\\
&=\frac{\varphi(\rho(g)v)\varphi(\rho(g^{-1})v)}{\varphi(v)^2}\\
&=\frac{\varphi(\rho(u_1^-\delta(g)u_1)v)\varphi(\rho(u_2^-\delta(g^{-1})u_2)v)}{\varphi(v)^2}\\
&=\frac{\varphi(\rho(\delta(g))v)\varphi(\rho(\delta(g^{-1}))v)}{\varphi(v)^2}\\
&=\chi(\delta(g))\chi(\delta(g^{-1})).
\end{align*}
The conclusion follows.
\end{proof}

\subsection{Proximal representations and loxodromic elements}
We now return to the real setting. 
In order to conclude the proof of Proposition \ref{asymptoticcrossratio}, we now relate the fact that an element $g$ of $G$ is loxodromic with the one that $\rho(g)$ is proximal for sufficiently many ${\bf P}$-proximal representations $\rho$ of ${\bf G}$. This is just a reformulation of the techniques used in \cite{Be1,Be2} (see also \cite[Chap. 8]{BQ}).

\begin{Lem}\label{checkproximal} 
Let $g$ be in $G$.
If $g$ is loxodromic, then, for any ${\bf P}$-proximal irreducible representation $\rho:{\bf G}\rightarrow {\rm GL}(V)$ with highest weight $\chi$, the endomorphism $\rho(g)$ of $V$ is proximal and $\lambda_1(\rho(g))=\chi(\lambda(g))$.

Conversely,  there exists a finite set of ${\bf P}$-proximal irreducible representations $\{\rho_1,\ldots,\rho_r\}$ such that, 
for $g$ in $G$,
if $\rho_1(g),\ldots,\rho_r(g)$ are proximal, then $g$ is loxodromic.
\end{Lem} 

\begin{proof} We can assume ${\bf A}\subset {\bf T}$ and, by \cite[Sect. 18.2]{Bo}, that ${\bf T}$ is defined over $\mathbb R$. 
We let $\Sigma_r\subset {\rm X}({\bf A})$ be the set of relative roots, that is, the set of weights of the adjoint action of ${\bf A}$ on the Lie algebra of ${\bf G}$. When choosing ${\bf P}$, we have chosen a positive system $\Sigma_a^+\subset\Sigma_a$. We write $\Pi_a\subset \Sigma_a^+$ for the associated basis. By \cite[Sect. 21.8]{Bo}, we can assume that the orders on the sets of absolute roots $\Sigma_a\subset {\rm X}({\bf T})$ and on the set of relative roots $\Sigma_r\subset {\rm X}({\bf A})$ are compatible in the following sense: if $\alpha$ is in $\Sigma_a^+$ and $\alpha$ has nonzero restriction to ${\bf A}$, then $\alpha_{|{\bf A}}$ belongs to $\Sigma_r^+$. Finally, we let $\Pi_i\subset \Pi_a$ (imaginary simple roots) be the set of simple roots with zero restriction to ${\bf A}$: the group ${\bf P}$ is the standard parabolic subgroup associated with $\Pi_i$.

Let $\rho:{\bf G}\rightarrow {\rm GL}(V)$ be a ${\bf P}$-proximal irreducible representation with highest weight $\chi$. Since ${\bf P}V_{\chi}\subset V_{\chi}$, the weights $\chi'\neq\chi$ of ${\bf T}$ in $V$ are of the form $\chi'=\chi-\alpha_1-\cdots -\alpha_k$, where $\alpha_1,\ldots,\alpha_k\in \Pi_a$ and $\alpha_1\not\in\Pi_i$ (and hence, the restriction of $\alpha_1$ to ${\bf A}$ is in $\Pi_r$). For $g$ in $G$, the eigenvalues of $\rho(g)$ in $V$ are the $\chi'(\lambda(g))$, where $\chi'$ runs among the weights of ${\bf T}$, and for such a $\chi'$, the modulus of $\chi'(\lambda(g))$ is given by $|\chi'(\lambda(g))|=e^{\de \chi'(\lambda_{\mathbb R}(g))}$ (with the notation of Section \ref{intro}). Saying that $g$ is loxodromic is saying that, for all $\beta\in\Pi_r$, one has $\de\beta(\lambda_{\mathbb R}(g))>0$ and hence, for all $\chi'\neq\chi$,  $|\chi'(\lambda(g))|<|\chi(\lambda(g))|$. Thus, $\chi(\lambda(g))$ is the unique eigenvalue of $\rho(g)$ with maximum modulus. Since $V_\chi$ is a line, this eigenvalue has multiplicity one. Therefore, $\rho(g)$ is proximal and $\lambda_1(\rho(g))=\chi(\lambda(g))$.

Conversely, for each $\beta$ in $\Pi_r$, there exists $\alpha$ in $\Pi_a$ such that $\alpha_{|{\bf A}}=\beta$. As in section \ref{proxrep}, let ${\bf P}_{\alpha}$ be the standard parabolic subgroup associated with $\Pi_{a}\setminus\{\alpha\}$. Then, since $\alpha\notin\Pi_i$, we have ${\bf P}\subset{\bf P}_{\alpha}$.
Since ${\bf P}_{\alpha}\neq {\bf G}$, the group ${\bf P}_{\alpha}$ admits characters which are not restrictions of characters of ${\bf G}$. By Corollary \ref{Levicharacter2}, we may therefore find a character $\chi_\alpha$ of ${\bf P}_{\alpha}$ which is the highest weight of a ${\bf P}_\alpha$-proximal irreducible representation 
$\rho_\alpha:{\bf G}\rightarrow {\rm GL}(V_\alpha)$, where $V_\alpha$ has dimension $>1$. Then, as the weight space $V_{\chi_\alpha}$ is ${\bf P}_\alpha$-invariant,
necessarily $\chi_{\alpha}-\alpha$ is a weight of ${\bf T}$ in $V_\alpha$. Therefore, for $g$ in $G$, if $\rho_\alpha(g)$ is proximal, we must have 
$\de\alpha(\lambda_{\mathbb R}(g))=\de\beta(\lambda_{\mathbb R}(g))>0$. Thus, to check that $g$ is loxodromic, it suffices to check that the finitely many $\rho_{\alpha} (g)$ are proximal.
\end{proof}

We now have all the tools in hand to conclude the

\begin{proof}[Proof of Proposition \ref{asymptoticcrossratio}]
This directly follows from Lemma \ref{asymptoticcrossratio1}, Corollary \ref{Levicharacter2}, Lemma \ref{representcrossratio} and Lemma \ref{checkproximal}.
\end{proof}

\section{Structure of the group $Z/[Z,Z]$}
\label{secstructureZAb}

We will now use Proposition \ref{asymptoticcrossratio} to prove Theorem \ref{densityJordan}. The strategy is to build points in $\mathcal P$ for which the cross ratio is very close to $e$ in $Z/[Z,Z]$ but not contained in any given proper closed subgroup of $Z/[Z,Z]$. To make this precise, we need a better understanding of the structure of the group $Z/[Z,Z]$.

We keep the previously introduced notation. In the absence of an adequate reference, we recall some classical facts on the relation between absolute and relative roots. As the group ${\bf T}$ is defined over $\mathbb R$ the complex conjugacy $g\mapsto \overline{g}$ leaves ${\bf T}$ stable and hence induces an automorphism of the root system $\Sigma_a$ which we still denote by $\alpha\mapsto\overline{\alpha}$. Denote by $\Sigma_i\subset\Sigma_a$ the subsystem spanned by $\Pi_i$ and set $\Sigma_i^+=\Sigma_i\cap\Sigma_a^+$. 

A root $\alpha$ in $\Sigma_a$ belongs to $\Sigma_i$ if and only if $\overline{\alpha}=-\alpha$. 
In particular, for $\alpha$ in $\Sigma_i$, the reflection $\sigma_{\alpha}$ commutes with the automorphism $x\mapsto\overline{x}$ of ${\rm X}_{\mathbb R}({\bf T})=\mathbb R\otimes_{\mathbb Z}{\rm X}({\bf T})$, hence, so does every element of the Weyl group $W_i$ of $\Sigma_i$. There exists a unique element $w_i$ of $W_i$ such that $w\Sigma_i^+=-\Sigma_i^+$. 

\begin{Lem}\label{roots1}
The automorphism $\alpha\mapsto w_i\overline{\alpha}=\overline{w_i\alpha}$ preserves $\Sigma_a^+$. 
\end{Lem}

\begin{proof} On one hand, if $\alpha$ is in $\Sigma_i^+$, we have $\overline{\alpha}=-\alpha$, hence $w_i\overline{\alpha}\in\Sigma_i^+$. 

On the other hand, for every $\beta$ in $\Pi_i$ and $\alpha$ in $\Sigma_a^+\setminus \Sigma_i^+$, we have $\sigma_\beta(\alpha)\in \Sigma_a^+\setminus \Sigma_i^+$. It follows that $w_i(\Sigma_a^+\setminus \Sigma_i^+)=(\Sigma_a^+\setminus \Sigma_i^+)$. Also, since $\alpha$ is not in $\Sigma_i^+$, the restriction $\gamma=\alpha_{|{\bf A}}$ of $\alpha$ to ${\bf A}$ is in $\Sigma_r^+$. As ${\bf A}$ is split, we have $\overline{\gamma}=\gamma$, hence $\gamma$ is also the restriction of $\overline{\alpha}$ to ${\bf A}$. This implies that $\overline{\alpha}$ belongs to $\Sigma_a^+\setminus \Sigma_i^+$ and therefore that $w_i\overline{\alpha}$ belongs to $\Sigma_a^+\setminus \Sigma_i^+$.
\end{proof} 

Note that since $w_i$ commutes with complex conjugacy, the map $\alpha\mapsto w_i\overline{\alpha}$ is an involution. We set 
$$r=|\{\alpha\in\Pi_a\setminus \Pi_i|w_i\overline{\alpha}={\alpha}\}|\mbox{ and }s=\frac{1}{2}|\{\alpha\in\Pi_a\setminus \Pi_i|w_i\overline{\alpha}\neq {\alpha}\}|.$$

\begin{Lem}\label{roots2} Every simple relative root in $\Pi_r$ is the restriction of a simple absolute root in $\Pi_a\setminus\Pi_i$. 
If $\alpha,\beta$ are in $\Pi_a\setminus\Pi_i$. Then, the simple roots $\alpha$ and $\beta$ have the same restriction to ${\bf A}$ if and only if $\beta=\alpha$ or $\beta=w_i\overline{\alpha}$. In particular, the real rank of ${\bf S}=[{\bf G},{\bf G}]$ is $r+s$.
\end{Lem}

\begin{proof} We can assume ${\bf G}$ to be semisimple.

Let $\alpha$ be in $\Pi_r$ and $\beta$ be a root in $\Sigma_a$ which restricts to $\alpha$. As $\alpha$ is positive, $\beta$ is in $\Sigma_a^+$. Write $\beta$ as a sum of simple roots, $\beta=\beta_1+\cdots+\beta_\ell$, with $\beta_1,\ldots,\beta_\ell\in\Pi_a$. We claim that there exists a unique $1\leq k\leq \ell$ with $\beta_k\in\Pi_a\setminus \Pi_i$. Indeed, if this were not the case, we could find $1\leq h,k\leq \ell$ with $\beta_h,\beta_k\in\Pi_a\setminus \Pi_i$. Then, up to shuffling the indices, there would exist $h\leq j< k$ such that both $\gamma_1=\beta_1+\cdots+\beta_j$ and $\gamma_2=\beta_{j+1}+\cdots+\beta_{\ell}$ would be roots. But then the restrictions $\alpha_1,\alpha_2$ of $\gamma_1,\gamma_2$ to ${\bf A}$ would be positive relative roots and one would have $\alpha=\alpha_1+\alpha_2$, a contradiction. Therefore, up to permuting the indices, we have $\beta_1\in\Pi_a\setminus \Pi_i$ and $\beta_2,\ldots,\beta_\ell\in\Pi_i$ and $\alpha$ is the restriction of $\beta_1$. This proves the first statement.

The other statements will be obtained by computing the dimension of the torus ${\bf A}$ in two ways.

On one hand, as ${\bf A}$ is the largest split subtorus of ${\bf T}$, the restriction to ${\bf A}$ of characters of ${\bf T}$
induces an isomorphism between the spaces ${\rm X}_{\mathbb R}({\bf A})=\mathbb R\otimes_{\mathbb Z}{\rm X}({\bf A})$ and $\{x\in {\rm X}_{\mathbb R}({\bf T})|\overline{x}=x\}$. This space is contained in the orthogonal complement $Y$ in ${\rm X}_{\mathbb R}({\bf T})$ of the subspace $X_i$ spanned by $\Sigma_i$ (where we have equipped ${\rm X}_{\mathbb R}({\bf T})$ with a scalar product invariant under the Weyl group and complex conjugacy). As $\Sigma_i$ is stable under complex conjugacy, so is $Y$. As $w_i$ is a product of reflections associated to elements of $\Sigma_i$, $w_i$ is the identity on $Y$. The dimension of ${\bf A}$ is therefore given by $\dim{\bf A}=\dim_{\mathbb R}\{x\in {\rm X}_{\mathbb R}({\bf T})/X_i|w_i\overline{x}=x\}$. Now, the space spanned by the roots in $\Pi_a\setminus\Pi_i$ is also a complementary subspace of $X_i$ which is stable under the involution $x\mapsto w_i\overline{x}$. Thus, we obtain $\dim{\bf A}=r+s$.

On the other hand we have shown above that the restriction map is a surjection $\Pi_a\setminus\Pi_i\rightarrow\Pi_r$. Besides, if $\alpha$ is in $\Pi_a\setminus\Pi_i$, then $\alpha$ and $w_i\overline{\alpha}$ have the same restriction to ${\bf A}$: indeed, $\alpha$ and $\overline{\alpha}$ have the same restriction and $w_i\overline{\alpha}-\overline{\alpha}$ is a sum of elements of $\Sigma_i$ which have trivial restriction. As we know that $|\Pi_r|=\dim{\bf A}=r+s$, the conclusion follows.
\end{proof}

We fix a set $\Delta\subset\Pi_a\setminus\Pi_i$ as follows: the set $\Delta$ contains every root $\alpha$ in $\Pi_a\setminus\Pi_i$ such that $w_i\overline{\alpha}=\alpha$; for every root $\alpha$ in $\Pi_a\setminus\Pi_i$ with $w_i\overline{\alpha}\neq\alpha$, the set $\Delta$ contains exactly one of the two simple roots $\alpha$ and $w_i\overline{\alpha}$. By Lemma \ref{roots2}, the restriction map is a bijection $\Delta \rightarrow\Pi_r$.

Assume now ${\bf G}$ is semisimple and simply connected.
As in Section \ref{proxrep}, for $\alpha$ in $\Pi_a$, we denote by $\chi_{\alpha}$ the associated fundamental weight. Since ${\bf G}$ is simply connected, the group of characters of ${\bf T}$ is ${\rm X}({\bf T})=\bigoplus_{\alpha\in\Pi_a}\mathbb Z\chi_{\alpha}$ and hence, by Proposition \ref{Levicharacter1}, the group of characters of ${\bf Z}$ is ${\rm X}({\bf Z})={\rm X}({\bf Z}/[{\bf Z},{\bf Z}])=\bigoplus_{\alpha\in\Pi_a\setminus\Pi_i}\mathbb Z\chi_{\alpha}$.

Using these tools, we can describe the group $Z/[Z,Z]$.

\begin{Prop} \label{ZAb}
Suppose ${\bf G}$ is semisimple and simply connected.

If $\alpha$ is in $\Pi_a\setminus\Pi_i$ and $w_i\overline{\alpha}=\alpha$, we have either $\chi_{\alpha}(Z)=\mathbb R^\star$ or $\chi_{\alpha}(Z)=\mathbb R^\star_+$.

If $\alpha$ is in $\Pi_a\setminus\Pi_i$ and $w_i\overline{\alpha}\neq\alpha$, we have $\chi_{\alpha}(Z)=\mathbb C^\star$.

The characters $(\chi_{\alpha})_{\alpha\in\Delta}$ define an isomorphism of real Lie groups between $Z/[Z,Z]$ and an open subgroup of 
$(\mathbb R^\star)^r\times (\mathbb C^\star)^s$.
\end{Prop}

\begin{proof} As the automorphism $\alpha\mapsto w_i\overline{\alpha}$ of the root system $\Sigma_a$ preserves the basis $\Pi_a$, for $\alpha$ in $\Pi_a$, we have $\chi_{w_i\overline{\alpha}}=w_i\overline{\chi_\alpha}$. Now, if $\alpha$ is in $\Pi_a\setminus\Pi_i$, 
the character $\chi_{\alpha}$ is orthogonal to $\Sigma_i$ hence $w_i\chi_{\alpha}=\chi_{\alpha}$, since $w_i$ is a product of reflections associated with elements of $\Sigma_i$. This yields $\overline{\chi_\alpha}=\chi_{w_i\overline{\alpha}}$. By Proposition \ref{Levicharacter1}, we have 
${\rm X}({\bf Z}/[{\bf Z},{\bf Z}])=\bigoplus_{\alpha\in\Pi_a\setminus\Pi_i}\mathbb Z\chi_{\alpha}$ and we have just shown the real form on this torus is defined by the automorphism which sends $\chi_{\alpha}$ to $\chi_{w_i\overline{\alpha}}$, $\alpha\in \Pi_a\setminus\Pi_i$. The conclusion now follows from elementary properties of algebraic groups (by using Lemma \ref{goodAbelian}).
\end{proof}

\begin{Ex} If $G={\rm SL}_2(\mathbb R)$, we have $Z/[Z,Z]\simeq \mathbb R^\star$. If $G={\rm SL}_2(\mathbb C)$, we have $Z/[Z,Z]\simeq \mathbb C^\star$. If $G={\rm SL}_2(\mathbb H)$, we have $Z/[Z,Z]\simeq \mathbb R_+^\times$. 
\end{Ex}

\section{Asymptotic expansion of cross ratios}

To conclude the proof of Theorem \ref{densityJordan}, we will build quadruples in $\mathcal W$ for which the cross ratio is close to the identity in $Z/[Z,Z]$ and its components in the decomposition of Proposition \ref{ZAb} have a nice asymptotic expansion. This is a direct extension of the method used in \cite{BQ,Q1}. The only change is that we need to control the full components in $\mathbb C^\star$ and not only their moduli.

\subsection{Zariski dense subsemigroups} Let $\Gamma\subset G$ be a Zariski dense subsemigroup. We recall some general facts.

The main result of \cite{P} is that $\Gamma$ contains a loxodromic element. The closure of the set of attractive fixed points in $\mathcal P$ of such elements is called the limit set of $\Gamma$ and denoted by $\Lambda_{\Gamma}$ (see \cite{Be1}). For brevity, we write $\Lambda_{\Gamma}^-=\Lambda_{\Gamma^{-1}}$. These sets are Zariski dense in $\mathcal P$.

As a consequence of Proposition \ref{asymptoticcrossratio}, we get:

\begin{Cor} \label{asymptoticcrossratio2} Let $\Gamma\subset G$ be a Zariski dense subsemigroup. Then, the closure in $Z/[Z,Z]$ of the set of elements of the form
$\lambda(gh)\lambda(g)^{-1}\lambda(h)^{-1}$, where $(g,h)$ are loxodromic elements of $\Gamma$ such that the product $gh$ is loxodromic, contains the set of cross ratios of elements of $\mathcal W\cap\Lambda_\Gamma\times \Lambda_{\Gamma}^-\times\Lambda_\Gamma\times \Lambda_{\Gamma}^-$.
\end{Cor}

\begin{proof} This follows from Proposition \ref{asymptoticcrossratio} and the fact that the set of pairs 
$(\xi_g^+,\xi_g^-)$, where $g$ is a loxodromic element of $\Gamma$, is dense in $\Lambda_\Gamma\times \Lambda_{\Gamma}^-$ (see \cite[Sect. 3.6]{Be1}).
\end{proof}

\subsection{Properties of the next eigenvalues} 
Assume that ${\bf G}$ is semisimple and simply connected and keep the notation of Section \ref{secstructureZAb}. Fix $\alpha$ in $\Pi_a\setminus\Pi_i$ and let $\rho_\alpha:{\bf G}\rightarrow{\rm GL}(V_{\alpha})$ be an irreducible representation with highest weight $\chi_{\alpha}$. We denote by $V_{\alpha}^+\subset V_{\alpha}$ the 
highest weight line, that is, $V_{\alpha}^+=\{v\in V_{\alpha} |\forall t\in {\bf T}\quad \rho_\alpha(t)v=\chi_\alpha(t)v\}$. 
We fix a vector $v_\alpha\neq 0$ in $V_\alpha^+$ and a linear functional $\varphi_\alpha\neq 0$ whose kernel is exactly the ${\bf T}$-invariant complementary subspace of $V_{\alpha^+}$. We assume $\varphi_\alpha(v_\alpha)=1$.

With the notation of Section \ref{intro}, if $g$ is a loxodromic element of $G$, then $\chi_{\alpha}(\lambda(g))$ is the unique eigenvalue of $\rho_\alpha(g)$ with maximal modulus. The associated eigenspace is the line $\rho_{\alpha}(\gamma^{-1})V_{\alpha}^+$ where $\gamma\in G$ is such that $\gamma g_h\gamma^{-1}\in\exp(\mathfrak a^{++})$. Now, we want to analyze the next eigenvalues of $\rho_\alpha(g)$.  

If $\chi\neq\chi_{\alpha}$ is another weight of ${\bf T}$ in $V_\alpha$, then $\chi$ may be written as $\chi_\alpha-\alpha-\beta_1-\ldots-\beta_r$ for some $\beta_1,\ldots,\beta_r$ in $\Pi_a$. We denote by $V_{\alpha}^{\pm}$ the sum of those weight spaces associated with weights $\chi$ for which the elements $\beta_1,\ldots,\beta_r$ belong to $\Pi_i$, that is, they are simple roots with trivial restriction to ${\bf A}$. We denote by $p_\alpha:V_\alpha\rightarrow V_{\alpha}^{\pm}$ the ${\bf T}$-equivariant projection.

If $g$ is a loxodromic element of $G$, the eigenvalues of $\rho_{\alpha}(g)$ with maximal modulus $<|\chi_{\alpha}(\lambda(g))|$ are exactly the $\chi(\lambda(g))$ where $\chi$ is as above. They have modulus 
$|\alpha(\lambda(g))|^{-1}|\chi_{\alpha}(\lambda(g))|$. In particular, if $g_h$ is in $\exp(\mathfrak a^{++})$ and $g$ is in $T$, there exists $\varepsilon>0$ such that,
for $v$ in $V_\alpha$, as $n\rightarrow\infty$, we have 
\begin{multline}\label{eqasymptoticexp1}
\chi_{\alpha}(g)^{-n}\rho_{\alpha}(g^n)v=\\
\varphi_{\alpha}(v)v_\alpha+\chi_{\alpha}(g)^{-n}\rho_{\alpha}(g^n)p_{\alpha}v+
O(e^{-\varepsilon n}|\alpha(g)|^{-n}).\end{multline}
Note that, in the above, the norm of the vector $\chi_{\alpha}(g)^{-n}\rho_{\alpha}(g^n)p_{\alpha}v$ is precisely of the order of $|\alpha(g)|^{-n}$. 

This asymptotic expansion lies at the core of the proof of Theorem \ref{densityJordan}. 
As in Section \ref{crossratios}, $\xi_0$ is the fixed point of ${\bf P}$ in $\mathcal P$ and $\eta_0$ is the one of ${\bf P}^-$.
To ensure that the second term of the right-hand side of \eqref{eqasymptoticexp1} may be chosen to be nonzero, we will use

\begin{Prop}\label{nonzerogeneric1}
Let ${\bf G}$ be semisimple and simply connected and let $\alpha$ be in $\Pi_a\setminus\Pi_i$. For $\xi\in {\bf U}^-\xi_0$, a flag in general position with $\eta_0$, and $\eta\in {\bf U}\eta_0$, a flag in general position with $\xi_0$, we set 
$$\theta_\alpha(\xi,\eta)=\frac{\varphi(p_\alpha v)}{\varphi_\alpha(v)\varphi(v_\alpha)},$$
where $v\neq 0$ is in $\rho_\alpha(\xi)$ and $\varphi\neq 0$ is in $\rho_\alpha^\vee(\eta)$.

If $w_i\overline{\alpha}=\alpha$, we have $\theta_\alpha(U^-\xi_0\times U\eta_0)=\mathbb R$.

If $w_i\overline{\alpha}\neq\alpha$, we have $\theta_\alpha(U^-\xi_0\times U\eta_0)=\mathbb C$.
\end{Prop}

Note that the function $\theta_\alpha$ is well-defined: indeed, we can choose $v$ (resp. $\varphi$) to be of the form $\rho_\alpha(u^-)v_\alpha$ (resp. $\rho_\alpha^\vee(u)\varphi_\alpha$) for some $u^-\in {\bf U}^-$ (resp. $u\in {\bf U}$) and we then get
$$\varphi_\alpha(v)=\varphi_\alpha(\rho_\alpha(u^-)v_\alpha)=1\mbox{ and }\varphi(v_\alpha)=\varphi_\alpha(\rho_\alpha(u)^{-1}v_\alpha)=1.$$

For $\beta$ in $\Sigma_a$, we denote by $\mathfrak g_\beta$ the root space associated with $\beta$ in the Lie algebra of ${\bf G}$.

\begin{proof}[Proof of Proposition \ref{nonzerogeneric1} in case $w_i\overline{\alpha}\neq \alpha$] 
We take $X$ in $\mathfrak g_\alpha$ and $Y$ in $\mathfrak g_{-\alpha}$ and we will compute $\theta_\alpha(\exp(Y+\overline{Y})\xi_0,\exp(X+\overline{X})\eta_0)$.
This computation relies mainly on the fact that the weights $\chi\neq\chi_\alpha$ of $\rho_\alpha$ are of the form $\chi_\alpha-\alpha-\beta_1-\cdots-\beta_r$ for $\beta_1,\ldots,\beta_r$ in $\Pi_a$.

Any product of at least two element among $\de\rho_\alpha(Y)$ and $\de\rho_\alpha(\overline{Y})$ maps $v_\alpha$ to a vector in a weight space associated with a weight $\chi$ such that the restriction of $\chi_\alpha-\chi$ to ${\bf A}$ is not $\alpha$. Therefore, we get 
$$p_\alpha(\rho_\alpha(\exp(Y+\overline{Y})) v_\alpha)=(\de\rho_\alpha(Y)+\de\rho_\alpha(\overline{Y}))v_\alpha.$$ 
As $\overline{\alpha}=w_i\alpha+\beta_1+\ldots+\beta_r$ for some $\beta_1,\ldots,\beta_r\in\Pi_i$, the character $\chi_\alpha-\overline{\alpha}$ is not a weight 
of $\rho_\alpha$, hence $\de\rho_\alpha(\overline{Y})v_\alpha=0$ and $p_\alpha(\rho_\alpha(\exp(Y+\overline{Y})) v_\alpha)=\de\rho_\alpha(Y)v_\alpha$.
Also $\chi_\alpha-\alpha+\overline{\alpha}$ is not a weight, so that
$$\rho_\alpha(\exp(X+\overline{X}))p_\alpha(\rho_\alpha(\exp(Y+\overline{Y})) v_\alpha)=
(1+\de\rho_\alpha(X))\de\rho_\alpha(Y)v_\alpha.$$
We obtain
\begin{multline*}\theta_\alpha(\exp(Y+\overline{Y})\xi_0, \exp(X+\overline{X})\eta_0)=\varphi_{\alpha}(\de\rho_\alpha(X)\de\rho_\alpha(Y)v_\alpha)
\\=\varphi_\alpha(\de\rho_\alpha([X,Y])v_\alpha)=\de\chi_\alpha([X,Y]),\end{multline*}
where we have used the fact that $\de\rho_\alpha(X)v_\alpha=0$ and that the Lie bracket $[X,Y]$ belongs to the Lie algebra of ${\bf T}$.
Note that for $X\neq 0$ and $Y\neq 0$, we have $[X,Y]\neq 0$ (see \cite[Sect. 13.18]{Bo}, \cite[Sect. 8.3]{Humalg}, \cite[Sect. 26.2]{Humgp}). 
As $\sigma_\alpha(\mathfrak g_\alpha)=\mathfrak g_{-\alpha}$, the line $[\mathfrak g_\alpha,\mathfrak g_{-\alpha}]$ is the eigenspace associated with the eigenvalue $-1$ of $\sigma_\alpha$ in the Lie algebra of ${\bf T}$ and hence, since $\chi_\alpha$ and $\alpha$ are not orthogonal, $\de\chi_\alpha$ does not vanish on this line. Thus, we have shown that $\theta_\alpha(U^-\xi_0\times U\eta_0)=\mathbb C$. 
\end{proof}

In the case where $w_i\overline{\alpha} = \alpha$ in Proposition \ref{nonzerogeneric1}, for showing that $\theta_\alpha(U^-\xi_0\times U\eta_0)\subset\mathbb R$, we will use the following fact which is actually a consequence of \cite{Tits}:

\begin{Lem} \label{realform}
Let ${\bf G}$ be semisimple and simply connected and let $\alpha$ be in $\Pi_a\setminus\Pi_i$. If $w_i\overline{\alpha}=\alpha$, then the representation $\rho_\alpha$ preserves a real form of $V_\alpha$. In other words, there exists a semilinear involution $v\mapsto\overline{v}$ of $V_\alpha$ such that, for any $g$ in ${\bf G}$ and $v$ in $V_\alpha$,
$$\overline{\rho_\alpha(g) v}=\rho_{\alpha}(\overline{g})\overline{v}.$$
\end{Lem}

\begin{proof} Recall that, by the proof of Proposition \ref{ZAb}, we have $\overline{\chi_\alpha}=\chi_\alpha$ 
Let $\overline{V_\alpha}$ be the conjugate space of $V_\alpha$, that is, $\overline{V_\alpha}$ admits the same underlying Abelian group as $V_\alpha$ but the action of $\mathbb C$ is now twisted by complex conjugacy. We define the representation $\overline{\rho_\alpha}:{\bf G}\rightarrow{\rm GL}(\overline{V_\alpha})$ as being given by, for $v$ in $\overline{V_\alpha}$, $\overline{\rho_\alpha}(g)v=\rho_\alpha(\overline{g})v$. Then, this representation is an irreducible rational representation of ${\bf G}$ in $\overline{V_\alpha}$. The Borel subgroup ${\bf B}\supset{\bf T}$ associated with the choice of the positive system $\Sigma_a^+\subset\Sigma_a$ is contained in the group ${\bf P}$ which is defined over $\mathbb R$. Hence, we get $\overline{\bf B}\subset{\bf P}$ and therefore, since $V_\alpha^+\subset V_\alpha$ is $\rho_\alpha({\bf P})$-invariant, the line $\overline{V_\alpha^+}\subset \overline{V_\alpha}$ is $\overline{\rho_\alpha}(\bf B)$-invariant. Thus it is the highest weight space of $\overline{\rho_\alpha}$ and, as $\overline{\chi_\alpha}=\chi_\alpha$, the highest weight of $\overline{\rho_\alpha}$ is $\chi_\alpha$. By the uniqueness theorem (see \cite[Sect. 31.3]{Humgp}), the two representations $\rho_\alpha$ and $\overline{\rho_\alpha}$ are conjugated. In other words, there exists a semilinear automorphism
$T:V_\alpha\rightarrow V_\alpha$ such that, for any $g$ in ${\bf G}$, $T\rho_\alpha(g)=\rho_\alpha(\overline{g})T$. As $T^2$ is a $\mathbb C$-linear automorphism which commutes with $\rho_\alpha$, by Schur's Lemma $T^2$ is a scalar $\mu$. By construction, the semilinear map $T$ preserves $V_{\alpha}^+$, hence there exists a complex number $\lambda$ with $Tv_\alpha=\lambda v_\alpha$. We obtain $T^2 v_\alpha =|\lambda|^2 v_\alpha$, hence, $\mu>0$. Up to replacing $T$ by $\mu^{-1/2}T$, we have $T^2=1$ and the lemma follows.
\end{proof}

We can now deal with the remaining case of the proposition.

\begin{proof}[Proof of Proposition \ref{nonzerogeneric1} in case $w_i\overline{\alpha}=\alpha$] 
Still pick $X$ in $\mathfrak g_\alpha$ and $Y$ in $\mathfrak g_{-\alpha}$. We will again compute $\theta_\alpha(\exp(Y+\overline{Y})\xi_0,\exp(X+\overline{X})\eta_0)$.

As before, any product of at least two element among $\de\rho_\alpha(Y)$ and $\de\rho_\alpha(\overline{Y})$ maps $v_\alpha$ to a vector in a weight space associated with a weight $\chi$ such that the restriction of $\chi_\alpha-\chi$ to ${\bf A}$ is not $\alpha$. Therefore, we get 
$$p_\alpha(\rho_\alpha(\exp(Y+\overline{Y})) v_\alpha)=(\de\rho_\alpha(Y+\overline{Y}))v_\alpha.$$ 
Reasoning in the same way, we obtain that 
\begin{multline*}\rho_\alpha(\exp(X+\overline{X}))p_\alpha(\rho_\alpha(\exp(Y+\overline{Y})) v_\alpha)=\\
(1+\de\rho_\alpha(X+\overline{X}))(\de\rho_\alpha(Y+\overline{Y}))v_\alpha,\end{multline*}
hence
\begin{multline}\label{computetheta1}
\theta_\alpha(\exp(Y+\overline{Y})\xi_0,\exp(X+\overline{X})\eta_0)=\\
\varphi_\alpha((\de\rho_\alpha(X+\overline{X}))(\de\rho_\alpha(Y+\overline{Y}))v_\alpha).\end{multline}

Now we again need to split the proof into two cases. First we assume that $\overline{\alpha}=\alpha$, so that $\overline{X}\in\mathfrak g_\alpha$ and 
$\overline{Y}\in\mathfrak g_{-\alpha}$. Then, \eqref{computetheta1} yields
\begin{multline*}
\theta_\alpha(\exp(Y)\xi_0,\exp(X)\eta_0)=
\varphi_\alpha(\de\rho_\alpha(X+\overline{X})\de\rho_\alpha(Y+\overline{Y})v_\alpha)\\
=\varphi_\alpha(\de\rho_\alpha([X+\overline{X},Y+\overline{Y}])v_\alpha)=\de\chi_\alpha([X+\overline{X},Y+\overline{Y}]),
\end{multline*}
where we have reasoned as in the first case. Using the properties of the Lie bracket $[\mathfrak g_{\alpha},\mathfrak g_{-\alpha}]$ in the same way,
we obtain $\mathbb R\subset\theta_\alpha(U^-\xi_0\times U\eta_0)$. 

Suppose now $\overline{\alpha}\neq\alpha$. As above, we write $\overline{\alpha}=\alpha+\beta_1+\ldots+\beta_r$ 
for some $r\geq 1$ and $\beta_1,\ldots,\beta_r\in\Pi_i$. Since neither $\chi_\alpha+\beta_1+\ldots+\beta_r$ nor $\chi_\alpha-\beta_1+\ldots+\beta_r$ are weights of $\rho_\alpha$, we get
$$\de\rho_\alpha(\overline{X})\de\rho_\alpha(Y)v_\alpha=\de\rho_\alpha(X)\de\rho_\alpha(\overline{Y})v_\alpha=0,$$
hence, by \eqref{computetheta1},
\begin{multline*}
\theta_\alpha(\exp(Y+\overline{Y})\xi_0,\exp(X+\overline{X})\eta_0)=\\
\varphi_\alpha((\de\rho_\alpha(X)\de\rho_\alpha(Y)+\de\rho_\alpha(\overline{X})\de\rho_\alpha(\overline{Y}))v_\alpha).\end{multline*}
Reasoning as above, this yields
\begin{multline*}
\theta_\alpha(\exp(Y+\overline{Y})\xi_0,\exp(X+\overline{X})\eta_0)=
\de\chi_{\alpha}([X,Y]+[\overline{X},\overline{Y}])\\=\de\chi_{\alpha}([X,Y]+\overline{[X,Y]}).\end{multline*}
As above, we get 
$\mathbb R\subset\theta_\alpha(U^-\xi_0\times U\eta_0)$. 

In both cases, to show the converse property, we let $v\mapsto\overline{v}$ be a real form of $V_\alpha$ as in Lemma \ref{realform}.
Since the weights of $\rho_\alpha$ are $\chi_\alpha$ and weights $\leq\chi_\alpha-\alpha$, the set of weights of the form $\chi_\alpha-\alpha-\beta_1-\ldots-\beta_r$ with $\beta_1,\ldots,\beta_r\in\Pi_i$ is invariant under conjugacy and therefore, the operator $p_\alpha$ commutes with conjugacy. Now, for $\xi\in U^-\xi_0$ and $\eta\in U\eta_0$, as $\overline{\xi}=\xi$ and $\overline{\eta}=\eta$, the lines $\rho_\alpha(\xi)\in\mathbb P(V_\alpha)$ and $\rho_\alpha^\vee(\eta)\in\mathbb P(V_\alpha^*)$ are real and we may choose generators $v$ and $\varphi$ of these lines which are real. Also, we can assume $v_\alpha$ and $\varphi_\alpha$ to be real. Then, $\theta_\alpha(\xi,\eta)=\frac{\varphi(p_\alpha v)}{\varphi_\alpha(v)\varphi(v_\alpha)}$ is real.
\end{proof}

\subsection{Conclusion} We now finish the proof of Theorem \ref{densityJordan}. We need one final auxiliary result. 

\begin{Lem}\label{toricmeasure} 
Let $H\simeq \mathbb R^d/\mathbb Z^d$ be a compact and connected Abelian Lie group and let $f:H\rightarrow \mathbb C$ be a function which is a finite sum of characters. Suppose there exists $h$ in $H$ with $f(h)\notin\mathbb R$. Then, the set $f^{-1}(\mathbb R)$ has zero Haar measure. 
\end{Lem}

\begin{proof} This follows from the fact that $f^{-1}(\mathbb R)$ is a proper real algebraic subset of $H$, when $H$ is viewed as a real algebraic subvariety in $(\mathbb C^\star)^d$ (see \cite[Chap. 3]{BoCoRo}). 
\end{proof}

We can now conclude the

\begin{proof}[Proof of Theorem \ref{densityJordan}] 
First, we assume that ${\bf G}$ is semisimple and simply connected (the general case will follow by elementary reductions). We prove the result by contradiction. Set $\Theta$ to be the closed subgroup of $Z/[Z,Z]$ spanned by the $\lambda(gh)\lambda(g)^{-1}\lambda(h)^{-1}$, where $g,h$ are loxodromic elements of $\Gamma$ such that $gh$ is loxodromic. Assume $\Theta$ is not an open subgroup of $Z/[Z,Z]$. Then, the Lie algebra of $\Theta$ is a proper subspace of the Lie algebra of $Z/[Z,Z]$. Since the connected component of $\Theta$ is the exponential of its Lie algebra, by Proposition \ref{ZAb}, we may find a family $(c_\alpha)_{\alpha\in\Delta}$ of complex numbers which are not all zero such that $c_\alpha\in \mathbb R$ when $w\overline{\alpha}=\alpha$ and, for any $z$ in $\Theta$, which is close enough to $e$,
$$\sum_{\alpha\in\Delta}\Re(c_\alpha\log\chi_{\alpha}(z))=0,$$
(where, if $w\overline{\alpha}\neq\alpha$, we have used the standard determination of the logarithm).
By Corollary \ref{asymptoticcrossratio2}, for all $(\xi,\eta,\xi',\eta')$ in $\mathcal W\cap\Lambda_\Gamma\times\Lambda_\Gamma^-\times\Lambda_\Gamma\times\Lambda_\Gamma^-$, if their cross ratio is close enough to $e$, we have
\begin{equation}\label{eqcrossratiosum}\sum_{\alpha\in\Delta}\Re(c_\alpha\log\chi_{\alpha}([\xi,\eta,\xi',\eta']))=0,\end{equation}
We will reach a contradiction by producing cross ratios on which the above sum does not vanish.

More precisely, by \cite[Lemma 7.21]{BQ}, we may find a loxodromic element $g$ in $\Gamma$ such that the numbers $\de\beta(\lambda_{\mathbb R})(g)$ are all distinct, when $\beta$ runs in the set of positive restricted roots $\Pi_r$. In other words, this says that the numbers $\alpha(\lambda(g))$ all have different moduli, when $\alpha$ runs in the set $\Delta$. Up to replacing $\Gamma$ by a conjugate, we can assume $g_h$ belongs to $\exp(\mathfrak a^{++})$ and $g$ belongs to $T$. Then, in particular, we have $\xi_0\in \Lambda_{\Gamma}$ and $\eta_0\in\Lambda_\Gamma^-$.

Let $\alpha\in\Delta$ be such that $c_\alpha\neq 0$ and that the number $|\alpha(g)|>1$ is minimal. As $\Lambda_\Gamma^-$ is Zariski dense in $\mathcal P$, we may choose $\eta$ in $\Lambda_\Gamma^-$ which is in general position with $\xi_0$ and such that $p_\alpha^*\rho_\alpha^\vee(\eta)\neq 0$. As $\Lambda_\Gamma$ is Zariski dense in $\mathcal P$, by Proposition \ref{nonzerogeneric1},
we may find $\xi$ in $\Lambda_\Gamma$ which is in general position with $\eta_0$ and $\eta$ and such that, for $v\neq 0$ in $\rho_\alpha(\xi)$ and $\varphi\neq 0$ in $\rho_\alpha^\vee(\eta)$, we have
\begin{equation}\label{nonzerogeneric}
\Re\left(c_\alpha \frac{\varphi (p_\alpha v)}{\varphi_\alpha(v)\varphi(v_\alpha)}\right)\neq 0.\end{equation}

Let $H\subset (\mathbb C^*)^{\Pi_i}$ be the closed subgroup spanned by the element $h=(\beta(g))_{\beta\in\Pi_i}$ (recall that for each $\beta$ in $\Pi_i$, we have $|\beta(g)|=1$). Up to replacing $g$ by a power, we may assume $H$ to be connected. Then, there exists a function $f:H\rightarrow \mathbb C$, which is a finite sum of characters, such that, for any $n\geq 0$,
$$f(h^n)=\frac{1}{\varphi_\alpha(v)\varphi(v_\alpha)}\frac{\alpha(g)^n}{\chi_{\alpha}(g)^n}\varphi(g^np_{\alpha}v).$$

For $\beta$ in $\Delta$, by Lemma \ref{representcrossratio}, we have
$$\chi_\beta([\xi_0,\eta_0,g^n\xi,\eta])=\frac{\varphi_\beta(\rho_\beta(g^n)v)\varphi(v_\beta)}{\varphi(\rho_\beta(g^n)v)}=
\frac{\varphi_\beta(v)\varphi(v_\beta)}{\varphi(\chi_{\beta}(g)^{-n}\rho_\beta(g^n)v)}.$$
Now, notice that, for $\beta\neq\alpha$ in $\Delta$ with $c_\beta\neq 0$, the assumption that the modulus of $\alpha(g)$ is minimal and \eqref{eqasymptoticexp1} yield, as $n\rightarrow\infty$, for some $\varepsilon>0$,
$$\log\chi_\beta([\xi_0,\eta_0,g^n\xi,\eta])=O(e^{-\varepsilon n}|\alpha(g)|^{-n}).$$
Besides, still by \eqref{eqasymptoticexp1}, we obtain
$$\log\chi_\alpha([\xi_0,\eta_0,g^n\xi,\eta])=-\alpha(g)^{-n}f(h^n)+O(e^{-\varepsilon n}|\alpha(g)|^{-n}).$$
By \eqref{nonzerogeneric}, we have $\Re(c_\alpha f(e))\neq 0$. Hence, by Lemma \ref{toricmeasure}, the set $f^{-1}(c_\alpha^{-1}i\mathbb R)$ has Haar measure zero in $H$. Since the sequence $(h^{n})_{n\geq 0}$ is equidistributed in $H$, we may find an infinite subset $S\subset\mathbb N$ such that, as $n\rightarrow\infty$, $n\in S$,
$|\Re(c_{\alpha} f(h^n)|$ is bounded below by a positive constant. 
Consequently, we obtain 
$$\left|\sum_{\beta\in\Delta}\Re(c_\beta\log\chi_{\beta}([\xi_0,\eta_0,g^n\xi,\eta]))\right|\underset{\substack{n\rightarrow\infty\\ n\in S}}\asymp
|\alpha(g)|^{-n},$$
which contradicts \eqref{eqcrossratiosum} and the result follows in case ${\bf G}$ is semisimple and simply connected.

In the general case, we let $\widetilde{\bf S}$ be the universal cover of ${\bf S}=[{\bf G},{\bf G}]$ and ${\bf C}\subset {\bf G}$ be the center of ${\bf G}$. As usual, we write $C={\bf C}(\mathbb R)$ and $\widetilde{S}=\widetilde{\bf S}(\mathbb R)$. The natural map ${\bf C}\times \widetilde{\bf S}\rightarrow {\bf G}$ is surjective and therefore the range of the map $C\times \widetilde{S}\rightarrow G$ is a normal subgroup $G'\subset G$ which is open for the locally compact topology and hence has finite index. Therefore, the Zariski closure of $\Gamma\cap G'$ is a finite index subgroup of ${\bf G}$. Since ${\bf G}$ is connected, $\Gamma\cap G'$ is still Zariski dense in ${\bf G}$. After having replaced $\Gamma$ by the projection onto $\widetilde{S}$ of the inverse image of $\Gamma\cap G'$ in $C\times\widetilde{S}$, we are brought back to the case where ${\bf G}$ is semisimple and simply connected.
\end{proof}


\begin{thebibliography}{10}

\bibitem{BoCoRo}
J. Bochnak, M. Coste and M.-F. Roy,
\emph{Real Algebraic Geometry},
Ergebnisse der Mathematik und ihrer Grenzgebiete 36,
Springer, Heidelberg, 2010.


\bibitem{Bo}
A. Borel,
\emph{Linear Algebraic Groups},
Graduate Texts in Mathematics 126,
Springer, New York, 1991.


\bibitem{Be1}
Y. Benoist,
Propri\'et\'es asymptotiques des groupes lin\'eaires. 
\emph{Geom. Funct. Anal.} \textbf{7} (1997), 1--47.


\bibitem{Be2}
Y. Benoist,
Propri\'et\'es asymptotiques des groupes lin\'eaires (II), 
\emph{Adv. Stud. Pure Math.} \textbf{26} (2000), 33--48.


\bibitem{BQ}
Y. Benoist and J.-F. Quint,
\emph{Random Walks on Reductive Groups},
Ergebnisse der Mathematik und ihrer Grenzgebiete 62,
Springer, Cham, 2016.


\bibitem{CG}
J.-P. Conze and Y. Guivarc'h,
Densit\'e d'orbites d'actions de groupes lin\'eaires et propri\'et\'es d'\'equidistribution de marches al\'eatoires,
in \emph{Rigidity in Dynamics and Geometry},
M. Burger and A. Iozzi (eds.),
Springer, Berlin, 2002, 39--76.


\bibitem{GQX}
I. Grama, J.-F. Quint and H. Xiao,
Local limit theorem for the operator norm of products of random matrices,
arXiv:2607.19873,  2026, 1--63. 


\bibitem{GR}
Y. Guivarc'h and A. Raugi,
Actions of large semigroups and random walks on isometric extensions of boundaries,
\emph{Ann. Sci. \'Ec. Norm. Sup\'er.} \textbf{40} (2007), 209--249.


\bibitem{Humalg}
J. E. Humphreys,
\emph{Introduction to Lie Algebras and Representation Theory},
Graduate Texts in Mathematics 9,
Springer, New York, 1972.


\bibitem{Humgp}
J. E. Humphreys,
\emph{Linear Algebraic Groups},
Graduate Texts in Mathematics 21,
Springer, New York, 1981.


\bibitem{P}
G. Prasad,
$\mathbb{R}$-regular elements in Zariski-dense subgroups,
\emph{Q. J. Math.} \textbf{45} (1994), 541--545.


\bibitem{Q1}
J.-F. Quint,
Groupes de Schottky et comptage,
\emph{Ann. Inst. Fourier (Grenoble)} \textbf{55} (2005), 373--429.


\bibitem{Tits}
J. Tits,
Repr\'esentations lin\'eaires irr\'eductibles d'un groupe r\'eductif sur un corps quelconque,
\emph{J. Reine Angew. Math.} \textbf{247} (1971), 196--220.


%
%
%
%
%
%
%
%
%
%
%
%

\end{thebibliography}
\end{document}